\documentclass[a4paper,oneside]{amsart}

\def\Wdn#1\wdn{\marginpar{\tiny #1}}
\long\def\WDN#1\wdn{[WDN: #1]\Wdn[Comment]\wdn}

\usepackage{amssymb}
\usepackage{amsmath}
\usepackage{amsfonts}
\usepackage{amsthm}
\usepackage{mathrsfs}
\usepackage{fix-cm}
\usepackage{graphicx}
\usepackage{amscd}

\usepackage{soul}
\sodef\spred{}{.2em}{.9em plus.4em}{1em plus.1em minus.1em}

\usepackage[latin1]{inputenc}
\usepackage[T1]{fontenc}

\usepackage[english]{babel}

\usepackage{bbm}

\usepackage[all]{xy}

\newbox\mybox
\def\overtag#1#2#3{\setbox\mybox\hbox{$#1$}\hbox to
  0pt{\vbox to 0pt{\vglue-#3\vglue-\ht\mybox\hbox to \wd\mybox
      {\hss$\ss#2$\hss}\vss}\hss}\box\mybox}
\def\undertag#1#2#3{\setbox\mybox\hbox{$#1$}\hbox to 0pt{\vbox to
    0pt{\vglue#3\vglue\ht\mybox\hbox to \wd\mybox
      {\hss$\ss#2$\hss}\vss}\hss}\box\mybox}
\def\lefttag#1#2#3{\hbox to 0pt{\vbox to 0pt{\vss\hbox to
      0pt{\hss$\ss#2$\hskip#3}\vss}}#1}
\def\righttag#1#2#3{\hbox to 0pt{\vbox to 0pt{\vss\hbox to
      0pt{\hskip#3$\ss#2$\hss}\vss}}#1}
\let\ss\scriptstyle

\def\Dot{\lower.2pc\hbox to 2.5pt{\hss$\bullet$\hss}}
\def\Circ{\lower.2pc\hbox to 2.5pt{\hss$\circ$\hss}}
\def\Vdots{\raise5pt\hbox{$\vdots$}}
\def\splicediag#1#2{\xymatrix@R=#1pt@C=#2pt@M=0pt@W=0pt@H=0pt}

\newcommand\lineto{\ar@{-}}
\newcommand\dashto{\ar@{--}}
\newcommand\dotto{\ar@{.}}

\newtheorem{thm}{Theorem}[section]
\newtheorem{lemma}[thm]{Lemma}
\newtheorem{prop}[thm]{Proposition}
\newtheorem{cor}[thm]{Corollary}

\newtheorem{thm*}{Theorem}

\theoremstyle{definition} 
\newtheorem{defn}[thm]{Definition} 
\newtheorem{ex}[thm]{Example}

\newcommand{\N}{\mathbbm{N}}
\newcommand{\C}{\mathbbm{C}}
\newcommand{\Q}{\mathbbm{Q}}

\newcommand{\R}{\mathbbm{R}}
\newcommand{\Z}{\mathbbm{Z}}

\newcommand{\E}{\mathbbm{E}}

\newcommand{\num}[1]{\lvert #1 \rvert}

\newcommand{\sign}{\operatorname{sign}}

\newcommand{\vx}{\operatorname{vert}}

\DeclareMathOperator*{\conect}{\#}

\newcommand{\morf}[4][\to]{ #2 \colon #3 #1 #4}

\newcommand{\inv}{^{-1}}

\makeatletter
\newcommand\@b@gconect[1]{%
\vcenter{\hbox{#1$\m@th\mkern2mu\conect\mkern2mu$}}}
\newcommand\@bigconect{%
\mathchoice{\@b@gconect\huge} 
{\@b@gconect\LARGE} 
{\@b@gconect{}} 
{\@b@gconect\footnotesize} 
}
\newcommand\bigconect{\mathop{\@bigconect}\displaylimits}
\makeatother

\renewcommand{\phi}{\varphi}
\renewcommand{\epsilon}{\varepsilon}

\begin{document}

\bibliographystyle{alpha}


\title[Splice singularities and universal abelian
cover of graph orbifolds]
{Splice diagram singularities and  the universal abelian
cover of graph orbifolds}
\author{Helge M\o{}ller Pedersen}
\address{Matematische Institut\\ Universität Heidelberg
\\ Heidelberg, 69120}
\email{pedersen@mathi.uni-heidelberg.de}
\begin{abstract}
  Given a rational homology sphere $M$, whose splice diagram
  $\Gamma(M)$ satisfy
  the semigroup condition, Neumann and Wahl were able to define a
  complete intersection surface singularity called splice diagram
  singularity from $\Gamma(M)$. They were also able to show that
  under an additional hypothesis on $M$ called the congruence
  condition, the link of the splice diagram singularity is the
  universal abelian cover of $M$. In this article we generalize the
  congruence condition to the
  class of orbifolds called graph orbifold. We show that under a small
  additional hypothesis, this orbifold congruence condition implies
  that the link or the splice diagram equations is the universal
  abelian cover. We also show that any two node splice diagram
  satisfying the semigroup condition, is the splice diagram of a graph
  orbifold satisfying the orbifold congruence condition. 
\end{abstract}
\maketitle

\section{Introduction}

The topology of an isolated complex surface singularity is
determined by its link, which all turn to be among the class of
$3$-manifolds called graph manifolds, that is the manifolds which only
have Seifert fibered pieces in their JSJ-decomposition. There are
several graph invariants of graph manifolds which is used to study
them. The first is the plumbing diagram of a plumbed $4$-manifold $X$
such that our graph manifold $M$ is the boundary of $X$, this does
give a complete invariant if one assumes the plumbing diagram is in a
normal form (see \cite{plumbing}), of which there are several
different. Now the plumbing diagrams can be quite large and does not
always display the properties of the manifold clearly, so we are
interested in an other invariant called splice diagram. Splice
diagrams where original introduced in \cite{EisenbudNeumann} and
\cite{Siebenmann} for integer homology sphere graph manifolds. It was
then later generalized by Neumann and Wahl to rational homology spheres in
\cite{NeumannWahl3}. They used it extensively in
\cite{neumannandwahl1} and \cite{neumannandwahl2}, and especially
their use in \cite{neumannandwahl2} is of interest to us.

In \cite{neumannandwahl2} they use the splice diagram of a singularity
link $M$ satisfying what they call the semigroup condition, to
construct a set of equations called splice diagram equations defining
an isolated complete intersection surface singularity $X$. They then
showed that if $M$ satisfy an additional hypothesis called the
congruence condition, the link of $X$ is the universal abelian cover
of $M$. In \cite{myarticle} I showed that the splice diagram of any
graph manifold $M$ always determines the universal abelian cover of
$M$. So combining these to result one gets a nice description of the
universal abelian cover of a graph manifold $M$, as the link of
complete intersection, provided that there is some graph manifold $M'$
with the same splice diagram satisfying the congruence condition. This
already implies that the congruence condition might not be needed,
moreover the following splice diagram 
$$\splicediag{8}{30}{
  &\Circ&&&&\Circ\\
  \Gamma=&&\Circ \lineto[ul]_(.5){3}
  \lineto[dl]^(.5){18}
  \lineto[rr]^(.25){23}^(.75){15}&& \Circ
  \lineto[ur]^(.5){2}
  \lineto[dr]_(.5){3}&\\
  &\Circ&&&&\Circ\hbox to 0 pt{~\hss} }$$ 
have no manifolds with it as its splice diagram satisfying the
congruence condition, even though it satisfy the semigroup condition and
it is the splice diagram of  $4$ different manifolds. Nonetheless one can
construct a plumbing diagram of the abelian cover use the algorithm derived
from my proof of Theorem 6.3 in \cite{myarticle} which is explained in
more detail in \cite{constructinabcovers} where this example is
explicitly constructed, and one can find a dual resolution graph for
the resolution of the splice diagram equation of $\Gamma$ by hand, and
it shows that also in this case is the link of the splice diagram
singularity the universal abelian cover. This again indicates that the
congruence condition is not needed. Even more interesting is the next
example. The following splice diagram
$$\splicediag{8}{30}{
  &\Circ&&&&\Circ\\
  \Gamma'=&&\Circ \lineto[ul]_(.5){3}
  \lineto[dl]^(.5){15}
  \lineto[rr]^(.25){23}^(.75){15}&& \Circ
  \lineto[ur]^(.5){2}
  \lineto[dr]_(.5){3}&\\
  &\Circ&&&&\Circ\hbox to 0 pt{~.\hss} }$$ 
satisfy the semigroup condition, and the universal abelian cover is
the link of the splice diagram singularity. But there are no manifolds
with $\Gamma'$ as its splice diagram. So what is the link the
universal abelian cover of? To prove Theorem 6.6 of \cite{myarticle} I
had to generalize the notion of splice diagram to a class of
$3$-dimensional orbifolds which I called graph orbifolds, and
$\Gamma'$ is the splice diagram of several graph orbifold.

This leads to the purpose of this article, to generalize the
congruence condition to graph orbifolds, which is done in Section
\ref{orbifoldcongruencecondition}. Show that, under a small extra
hypothesis, the link of at splice diagram equation of $\Gamma(M)$ is
the universal abelian cover of $M$ if $M$ satisfy the orbifold
congruence condition in Section \ref{universalabcoveroforbifolds}. In
section \ref{constructingorbifolds} we show that this is indeed an
extension of the results of \cite{neumannandwahl2}, by given any two
nodes splice diagram $\Gamma$ satisfying the semigroup condition constructing a
graph orbifold $M$ satisfying the orbifold congruence condition, with
$\Gamma$ as its splice diagram. Section \ref{graphorbifolds}
introduces graph orbifolds, Section \ref{splicediagrams} introduces
splice diagram, and Section \ref{spliceeq} the splice diagram
equations. Section \ref{plumninganddisc} introduces the discriminant
group which is needed in the definition of the congruence condition.

\section{Graph Orbifolds}\label{graphorbifolds}

To generalize the conditions for the splice diagram equation to define
the universal abelian cover, we have to extend the notion of splice
diagrams to graph orbifolds, so in this section we define
graph orbifolds.

\begin{defn}
A \emph{graph orbifold} is a 3-dimensional orbifold $M$, in which there
exist a finite collection $\{T_i\}_{i\in 1,\dots,n}$ of smoothly
embedded tori, such that
$M-\bigcup_{i=1}^nT_i$ is a collection of $S^1$ orbifold fibrations
over orbifold surfaces.
\end{defn}

We will only consider graph orbifolds which has compact closure which
is also a graph orbifold, and
the boundary components will always be smoothly embedded tori. Notice
that if $M$
is smooth then $M$ is a graph manifold, since any $S^1$ orbifold
fibration over an orbifold surface with smooth total space is Seifert
fibered. 

Next we want to describe how a graph orbifold look locally, so we will
look at a $S^1$ orbifold fibration over orbifold surface
$\morf{\pi}{M}{\Sigma}$. If 
$x\in\Sigma$ is not an orbifold point then there is a disk
neighborhood $D$ of $x$ such that $M\vert_{\pi\inv(x)}$ is a
trivial fibered solid torus. So the interesting situation is when $x\in\Sigma$
is an orbifold point. This means that a neighborhood $U$ of $x$ is
homeomorphic to $\R^2/(\Z/\alpha\Z)$ for a $\alpha>1$, where the group acts as
rotation. Then $M\vert_{\pi\inv(U)}$ is homeomorphic to $(U\times
S^1)/(\Z/\alpha\Z)$ where $[k]\in\Z/\alpha\Z$ acts by $[k](z,s)=(e^{2\pi
  ikv/\alpha}z,e^{-2\pi ikq/\alpha}s)$ where we consider
$U=\C$, $q\mid\alpha$ and $\gcd(v,\alpha)=1$. Let $\beta'$ satisfy
$\beta'v\equiv -1\mod \alpha$, and let $\beta=\beta'q$.
We call the pair $(\alpha,\beta)$ the \emph{Seifert
  invariants} of the singular fiber over $x$ if we choose
$0\leq\beta<\alpha$. Notice if 
$\gcd(\alpha,\beta)=1$ then the action on $(U\times S^1)$ is free and
hence $\pi\inv(U)$ is smooth. We call $q=\gcd(\alpha,\beta)$ for the
\emph{orbifold degree} of the singular fiber. We will say that any non
singular fiber has orbifold degree $1$, which off course follows from
the Seifert invariants of a non singular fiber being $(1,0)$. We denote by
$N_{(\alpha,\beta)}$ a solid torus neighborhood of the singular fiber
with Seifert invariants $(\alpha,\beta)$, and by abuse of notation we
call $N_q$ a solid torus neighborhood of a singular fiber of orbifold
degree $q$.

One can now construct an unique decomposition of a compact graph orbifold $M$,
the following way. Let $q_1,\dots,q_n$ be the orbifold degrees of all
the singular fibers which has orbifold degree greater than $1$, and let
$N_{q_i}$ be solid torus neighborhoods of the corresponding singular
fiber as above. Then $M'=M-\bigcup_{i=1}^nN_{q_i}$ is a graph manifold with
boundary, and hence the JSJ-decomposition of $M'$ gives us a
collection of Seifert fibered manifolds with boundary $M_j'$'s. We
then makes the pieces of the decomposition of $M$ by gluing back the
$N_{q_i}$ in their original places. This is unique since the JSJ
decomposition of $M'$ is unique. We will call this decomposition for
the \emph{JSJ-decomposition} of $M$. 

Let $K\subset M$ be an orbifold curve with Seifert invariants
$(\alpha,\beta)$, then the gluing of $N_{(\alpha,\beta)}$ into
$M'=M-N_{(\alpha,\beta)}$ is completely determined by a simple closed
curve in $T=\partial M'$ of slope $\beta/\alpha$. Let
$q=\gcd(\alpha,\beta)$ be the orbifold degree of $K$ and let
$\alpha'=\alpha/q$ and $\beta'=\beta/q$. Then $(\alpha',\beta')$
determines a Seifert fibration of the solid torus
$N_{(\alpha',\beta')}$, notice that the Seifert fibration of
$N_{(\alpha',\beta')}$ and the orbifold structure on
$N_{(\alpha,\beta)}$ determines the same closed curve in $\partial
M'$ since $\beta'/\alpha'=\beta/\alpha$, therefore if $M_K=M'\bigcup
N_{(\alpha',\beta')}$ using the same gluing, $M_K$ and $M$ has the
same topology, even though $M$ and $M_K$ is not equal as graph
orbifolds since the singular fiber $K$ has different Seifert
invariants. If one replaces $M$ with $M_{K_i}$ for all curves $K_i$
which has orbifold degree greater than one, we get a graph manifold
$\overline{M}$ which have the same topology as $M$, we call
$\overline{M}$ the \emph{underlying manifold of $M$}.  

In the definition of splice diagram of a graph manifold $M$, one uses
the order for the first singular homology group of some graph
manifolds constructed from $M$. Now the above implies that if we just
use the singular homology groups when we define the splice diagram of
a graph orbifold $M$, then one just get the splice diagram of the
underlying manifold. So if using orbifold to extend the theory has to
have any meaning, we need another homology theory, which reduces to
singular homology in $M$ is smooth, but sees the orbifold structure
otherwise. This is going to be what we will call \emph{orbifold
  homology}. Since we only need the first orbifold homology group, it
suffices to define $H^{orb}_1(M)$ to be the abelianazation of the
orbifold fundamental group $\pi^{orb}_1(M)$. The orbifold fundamental
group is a more classical object, and have been studied a lot,
especially in the geometrization of $3$ dimensional orbifolds (see e.g.\
\cite{scott}). Since $\pi^{orb}(M)=\pi(M)$ if $M$ is smooth it follows
that $H^{orb}_1(M)=H_1(M)$ if $M$ is smooth. 

\begin{prop}
 $H^{orb}_1(T_{\alpha,\beta})=\Z\oplus\Z/q\Z$ where $q$ is
  the orbifold degree of the central fiber, i.e.\
  $\gcd(\alpha,\beta)$.
\end{prop}

\begin{proof}
Since $\pi^{orb}_1(T_{\alpha,\beta})$ can be presented as $\langle
q,h,t\ \vert\ q^\alpha h^\beta=1,\ q^{\alpha'}h^{\beta'}=t,\
[q,h]=1\rangle$ where $\alpha\beta'-\beta\alpha'=q$, it is not hard to
show that that is indeed a presentation of $\Z\oplus\Z/q\Z$.
\end{proof}

This can then be used to show the following relating orbifold homology
of $M$ and $M_K$ defined above which is Proposition 5.2 of \cite{myarticle}

\begin{prop}\label{homologycomparison}
Let $K\in M$ be an orbifold curve of degree $q$, then
$\num{H^{orb}_1(M)}=q\num{H^{orb}_1(M_k)}$.
\end{prop}

It then follows that
$\num{H^{orb}_1(M)}=Q\num{H^{orb}_1(\overline{M})}= Q\num{H_1(M)}$,
where $Q=\prod_K q_K$ with the product is taking over all orbifold curves
$K$ of orbifold degree $q_K$. It should be mentioned that our
$H^{orb}_1(M)$ is part of a homology theory for orbifolds in general
see \cite{orbifold}.

\section{Splice Diagrams}\label{splicediagrams}

A \emph{splice diagram} is a tree with no vertices of valence $2$, which is
decorated by signs on vertices who have valence greater than $2$, we
call such vertices \emph{nodes}, and non negative integer weights on
edges adjacent to nodes. We will call vertices of valence $1$ for
\emph{leaves}, and will in general not distinguish between a leaf and
the edge connecting the leaf to a node.  

We will now explain how to assign a splice diagram as an invariant to
any graph orbifold. Let $M$ be a rational homology sphere ($\Q$HS) graph
orbifold, we the construct the splice diagram $\Gamma(M)$, by first
taking a node for each piece of the JSJ-decomposition of $M$, the
connect two nodes if the corresponding $S^1$-fibered pieces of the
decomposition are 
glued to create $M$. We will in general not distinguish between a node
in $\Gamma(M)$ and the corresponding $S^1$-fibered piece. This will
result in a tree since the pieces of the $JSJ$-decomposition of $M$
corresponds to the pieces of the $JSJ$-decomposition of the underlying
manifold $\overline{M}$, and since $M$ is a $\Q$HS it follows from
the comments after Proposition \ref{homologycomparison}, that
$\overline{M}$ is a $\Q$HS, and hence its $JSJ$-decomposition gives a
tree like structure. We add a leaf to a node for each singular fiber
of the node. 

Next we want to add the decorations. First the signs at a node $v$ is
going to be the sign of the linking number of two non singular fibers
at $v$, for precise definition of this see section $2$ of
\cite{myarticle}, there we do only define it for manifolds, but the
definition carries over to orbifolds, even though one has to be
careful since linking number is not going to be symmetric any more in
the case of orbifolds. One can calculate the linking number in
the underlying manifold, and it follows from Lemma
\ref{dotproductcomparison} that signs will be the same. 

Last thing to define is the edge weights. Let $v$ be a node and $e$ an
edge adjacent to $v$, then we will do the following construction to
define the edge weight $d_{ve}$ adjacent to $v$ on $e$. Let $T\subset
M$ be the torus corresponding to the edge $e$, cut $M$ along $T$ and
let $M'$ be the piece not containing $v$. Let
$M_{ve}=M'\bigcup_T(S^1\times D^2)$ by gluing a meridian of the solid
torus to the image of a fiber of $v$ in $\partial M'$. Then $d_{ve}=
\num{H^{orb}_1(M_{ve})}$.    

The above definition defines a splice diagram where all weight at
leaves are greater than $1$, we call such splice diagram
\emph{reduced}, we will not always assume that our splice diagram are
reduced, i.e.\ we will sometimes allow weights at leaf to be $1$. A leaf
of weight $1$ will correspond to a non singular fiber, so if one has a
non reduced splice diagram of a graph orbifold $M$, one gets the reduced
splice diagram of $M$ by removing all leaves of weight $1$.  

A general edge between two nodes in a splice diagram looks like
$$\splicediag{8}{30}{
  &&&&\\
  \Vdots&\overtag\Circ {v_0} {8pt}\lineto[ul]_(.5){n_{01}}
  \lineto[dl]^(.5){n_{0k_0}}
  \lineto[rr]^(.25){r_0}^(.75){r_1}&& \overtag\Circ{v_1}{8pt}
  \lineto[ur]^(.5){n_{11}}
  \lineto[dr]_(.5){n_{1k_1}}&\Vdots\\
  &&&&\hbox to 0 pt{~.\hss} }$$ 
to such an edge we associate a number called \emph{the edge
  determinant}, which we define as $r_0r_1- \epsilon_0\epsilon_1
\big(\prod_{i=1}^{k_0}n_{oi}\big)\big(\prod_{j=1}^{k_1}n_{1j}\big)$,
where $\epsilon_i$ is the sign on the $i$'th node. This
number is important since it helps determining which graph manifolds
arises as singularity links, by the following theorem from
\cite{myarticle}

\begin{thm}
Let $M$ be a $\Q$HS graph manifold with splice diagram
$\Gamma(M)$. Then $M$ is the link of an isolated complex surface
singularity if and only if, there are no negative signs on nodes and
all edge determinants of $\Gamma(M)$ are positive. 
\end{thm}

Not all combinatorial splice diagram arises as the splice diagram of
a graph orbifold. We will next introduce an important condition which
the splice diagram of any graph orbifold satisfy. Let $\Gamma$ be a
splice diagram, and let $v$ and $w$ be two vertices in $\Gamma$. Then one
defines the \emph{linking number} $l_{vw}$ of $v$ and $w$  to be the
product of all edge weights adjacent to but not on the shortest path
from $v$ to $w$. Similarly $l'_{vw}$ is defined the the same way except
we exclude the weights adjacent to $v$ and $w$. Let $e$ be an edge at
$v$ then $\Gamma_{ve}$ is the connected subgraph of $\Gamma$ one get
by removing $v$, which includes the edge $e$. We then define
\emph{ideal generator} at $v$ in direction of $e$ $\overline{d}_{ve}$,
to be the positive generator of the following ideal in $\Z$
\begin{align*}
\langle l'_{vw}\ \vert\ w\text{ is a leaf of }\Gamma_{ve}\rangle 
\end{align*} 

\begin{defn} A splice diagram $\Gamma$ satisfy \emph{the ideal
    condition} if every edge weight $d_{ve}$ is divisible by the
  corresponding ideal generator $\overline{d}_{ve}$
\end{defn}
In section 12.1 in \cite{neumannandwahl2} Neumann and Wahl proves that
every splice diagram coming from a singularity link satisfy the ideal
condition, and the proof also works in the general setting of any
graph orbifold. So satisfying the ideal condition is a necessary
condition for a splice diagram to come from graph orbifold, in the one
node case the ideal condition is void and we show in section
\ref{constructingorbifolds} the in the two node case the ideal
condition is also sufficient, unfortunately this is not the case if
the splice diagram have more than two nodes, and the lack of
understanding which splice diagram that in general arises as
invariant of graph orbifolds is one of the reasons we have not been
able to extend the main result to more than two nodes splice
diagrams. 

\section{Splice Diagram Equations}\label{spliceeq} 

We are in general going to be be interested in splice diagram
satisfying the following stronger condition than the ideal condition. 
\begin{defn}
A splice diagram $\Gamma$ is said to satisfy \emph{the semigroup
  condition} if for every node $v$ and edge $e$ at $v$. The edge
weight lies in the following semigroup of $\N$:
\begin{align*}
d_{ve}\in\N\langle l'_{vw}\ \vert\ w\text{ a leaf in
}\Gamma_{ve}\rangle
\end{align*}
\end{defn}
The semigroup is only interesting if one has no negative signs at
nodes, so we will assume that splice diagrams satisfying the semigroup
satisfy this. The semigroup condition is strictly stronger than the
ideal condition for splice diagram with more than one node. 

If $M$ satisfy the semigroup condition, then given a node $v$ and
adjacent edge $e$, one can write the corresponding edge weight as
\begin{align}
d_{ve}=\sum_{w\text{ is a leaf of
  }\Gamma_{ve}}\alpha_{vw}l'_{vw},\label{semigroupcoef}
\end{align}
where the $\alpha_{vw}$'s are non negative integers. We call the
collection of $\alpha_{vw}$ \emph{semigroup coefficients} of
$d_{ve}$. It not hard to see that \eqref{semigroupcoef} is equivalent to
\begin{align}
d_v=\sum_{w\text{ is a leaf of
  }\Gamma_{ve}}\alpha_{vw}l_{vw},
\end{align}
where $d_v$ is the product of all edge weights adjacent to $v$.

From now on we will assume that $\Gamma$ satisfy the semigroup
condition, we then associate to each leaf $w$ of $\Gamma$ a variable
$z_w$, and have the following definitions.

\begin{defn}
Let $v$ be a node of $\Gamma$ and assume $\Gamma$ has leaves
$w_1,\dots,w_n$, then a $v$-\emph{weighting}, or 
$v$-\emph{filtration}, of $\C [z_{w_1},\dots,z_{w_n}]$ is to associate
the weight $l_{vw_i}$ to $z_{w_i}$.
\end{defn}

\begin{defn}
Let $v$ be a node of $\Gamma$ an $e$ an adjacent edge, then an
\emph{admissible monomial} associated to $v$ and $e$, is a monomial on
the form $M_{ve}=\prod_w z_w^{\alpha_w}$, where the product is taken
over all leaves in $\Gamma_{ve}$, and the $\alpha_w$'s is a choice of
semigroup coefficients of $d_{ve}$.  
\end{defn}
It is clear that each admissible monomial is $v$-weighted homogeneous,
of total $v$-weight $d_v$. 

\begin{defn}
Let $\Gamma$ with leaves $w_1,\dots,w_n$ be a splice diagram
satisfying the semigroup condition, then a set of \emph{splice diagram
  equations} for $\Gamma$, is the following set of equations in the
variables $z_{w_1},\dots, z_{w_n}$. 
\begin{align*}
\sum_ea_{vie}M_{ve}+H_{vi},\ v\text{ a node of valence }\delta_v,\
e\text{ an adjacent edge},\ i=1,\dots,\delta_v-2
\end{align*}
where
\begin{itemize}
\item $M_{ve}$ is an admissible monomial
\item for every $v$, all maximal minors of the
  $((\delta_v-2)\times\delta_v)$-matrix $(a_{vie})$ has full rank   
\item $H_{vi}$ is a convergent power series in the $z_{w_i}$'s all of
  whose monomials has $v$-weight higher that $d_v$.
\end{itemize}
This defines $n-2$ equations in $n$ variables, and the corresponding
subscheme $X(\Gamma)\subset \C^n$ is called a \emph{splice diagram
  surface singularity}.
\end{defn}

We have the first important result concerning splice diagram
equations, which is Theorem 2.6 of \cite{neumannandwahl2}.

\begin{thm}
Let $\Gamma$ be a splice diagram satisfy the semigroup condition, and
let $X=X(\Gamma)$ be an associated splice diagram surface
singularity. Then $X$ is a two-dimensional complete intersection,
with an isolated singularity at the origin. 
\end{thm}

\section{Plumbing and The Discriminant Group}\label{plumninganddisc}

From now $M$ will always be a graph orbifold with splice
diagram $\Gamma(M)$ and $\overline{M}$ is the underlying manifold,
hence $\Gamma(\overline{M})$ is equal to $\Gamma(M)$ whit any
edge-weight $d_{ve}$, replaced by $d_{ve}/o_1o_2\cdots o_n$, where
$o_1,\dots,o_n$ are the orbifold degrees of all orbifold curves in
$M_{ve}$ by Proposition \ref{homologycomparison},  

Now this of course do not always produce an reduced splice diagram of
$\overline{M}$, since there could be leafs with weight $1$. 

\begin{ex}\label{ex1}
Assume $M$ is a graph manifold with following splice diagram 
$$\splicediag{8}{30}{
 &&&&&&\Circ\\
 &\overtag\Circ {w_1} {8pt}&&&&\Circ \lineto[ur]^(.5){2} \lineto[dr]_(.5){6}&\\
 \Gamma(M)=& &\Circ \lineto[ul]_(.5){6}
  \lineto[dl]^(.5){3}
  \lineto[rr]^(.25){4330}^(.75){78}&& \Circ
  \lineto[ur]^(.25){102}^(.75){4755}
  \lineto[dr]_(.5){5}&&\overtag\Circ {w_2} {8pt}\\
&\Circ&&&&\overtag\Circ {w_3} {8pt}&
\hbox to 0 pt{~,\hss} }$$ 
and assume that at the leaves named $w_1,w_2$ and $w_3$, we have
orbifold curves of orbifold degrees $3,2$ and $5$ respectively. Then
$\overline{M}$ is going to have the following splice diagram
$$\splicediag{8}{30}{
 &&&&&&\Circ\\
 &\overtag\Circ {w_1} {8pt}&&&&\Circ \lineto[ur]^(.5){2} \lineto[dr]_(.5){3}&\\
 \Gamma(\overline{M})=& &\Circ \lineto[ul]_(.5){2}
  \lineto[dl]^(.5){3}
  \lineto[rr]^(.25){433}^(.75){26}&& \Circ
  \lineto[ur]^(.25){56}^(.75){317}
  \lineto[dr]_(.5){1}&&\overtag\Circ {w_2} {8pt}\\
&\Circ&&&&\overtag\Circ {w_3} {8pt}&
\hbox to 0 pt{~.\hss} }$$ 
If we want the reduced splice diagram of $\overline{M}$, we just
remove the leaf $w_3$, and the the central vertex becomes a valence
two vertex, and therefore has to be suppressed, and we get
$$\splicediag{8}{30}{ 
 &\overtag\Circ {w_1} {8pt}&&&&\Circ \\
 \Gamma'(\overline{M})=& &\Circ \lineto[ul]_(.5){2}
  \lineto[dl]^(.5){3}
  \lineto[rr]^(.25){433}^(.75){317}&& \Circ
  \lineto[ur]^(.5){2}
  \lineto[dr]_(.5){3}&\\
&\Circ&&&&\overtag\Circ {w_2} {8pt}
\hbox to 0 pt{~.\hss} }$$ 
\end{ex}

Let
$\Delta(\overline{M})$ be a plumbing diagram of $\overline{M}$, then
one can get a plumbing diagram $\Delta(M)$ for $M$, by adding a arrow
weighted with the orbifold degree at the corresponding vertex for each
orbifold curve of $M$. By blowing up if necessary, we can assume that
each vertex has at most one arrow attached. If $v$ is a vertex of
$\Delta(M)$, then the orbifold degree $o_v$ of $v$ is the orbifold
degree of the arrow attached to $v$, if no arrow is attached to $v$ then
$o_v=1$. When ever we use that notation $\Delta(\overline{M})$ and
$\Delta(M)$, we will assume that they are connected in this way. 

\begin{ex}
The graph orbifold $M$ from Example \ref{ex1} has the following
plumbing diagram
$$
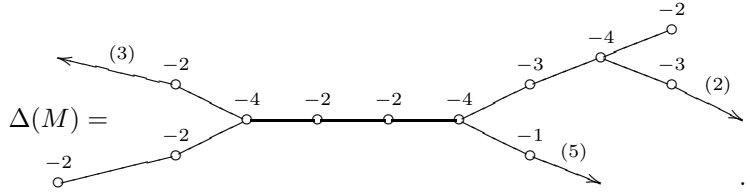
\xymatrix@R=6pt@C=24pt@M=0pt@W=0pt@H=0pt{
&&&&&&&& \overtag{\Circ}{-2}{8pt}&\\
&&&&&&& \overtag{\Circ}{-4}{8pt}\lineto[ur] \lineto[dr]&&\\
& \overtag{\Circ}{-2}{8pt}\ar[ul]_{(3)} &&&&& \overtag{\Circ}{-3}{8pt}
\lineto[ur] && \overtag{\Circ}{-3}{8pt}\ar[dr]^{(2)}&\\
\Delta(M)=&&\overtag{\Circ}{-4}{8pt}
\lineto[r]\lineto[ul]\lineto[dl]& \overtag{\Circ}{-2}{8pt}\lineto[r] &
\overtag{\Circ}{-2}{8pt}\lineto[r] & 
\overtag{\Circ}{-4}{8pt}\lineto[ur]\lineto[dr]&&&&\\
& \overtag{\Circ}{-2}{8pt}\lineto[dl] &&&&&
\overtag{\Circ}{-1}{8pt}\ar[dr]^{(5)} &&&\\
\overtag{\Circ}{-2}{8pt}&&&&&&&&&.}$$
\end{ex}

Let $\overline{E}_v\subset\overline{X}$ be the curve where $\overline{X}$ is
an analytic surface with $\partial\overline{X}=\overline{M}$ corresponding to
the vertex $v\in\Delta(\overline{M})$ and $E_v$ the corresponding
surface in $X$ with $\partial X=M$. Then let 

\begin{align}
\E\colon &=\bigoplus_{v\in vert(\Delta(M))}\Z\cdot E_v
\end{align}
and
\begin{align}
\overline{\E}\colon &=\bigoplus_{v\in vert(\Delta(M))}\Z\cdot \overline{E}_v
\end{align} 
Now $\E$ and $\overline{\E}$ are the same as $\Z$-modules, but they
have different intersection pairings defined by
$\Delta(\overline{M})$ and $\Delta(M)$. Let $A(\overline{M})$ be the
intersection matrix of $\overline{\E}$ in the basis given by
$\overline{E}_v$. Then the intersection matrix for $\E$ in the basis
$E_v$ is given by the matrix $A(M)$ which is gotten from
$A(\overline{M})$ by multiplying each column, which corresponds to a
vertex $v$, with the orbifold degree $o_v$ of $v$ as defined
above. 

One can construct $\Gamma(M)$ from $\Delta(M)$ by suppressing all
vertices of valence $2$ in $\Delta(M)$ to get the tree structure, and
using the following propositions from \cite{myarticle} to get the
decorations. 

\begin{prop}
Let $v$ be a node in $\Gamma(M)$, and $e$ be a edge on that node. We
get the weight $d_{ve}$ on that edge by
$d_{ve}=\num{\det(-A(\Delta(M)_{ve}))}$, where $\Delta(M)_{ve}$ is is the
connected component of $\Delta(M)-e$ which does not contain $v$.
\end{prop}

$$
\xymatrix@R=6pt@C=24pt@M=0pt@W=0pt@H=0pt{
&&&&&&&\\
\Delta(M)=&\Vdots&&\undertag{\overtag{\Circ}{a_{vv}}{8pt}}{v}{8pt}
\lineto[rr]\dashto[ull]\dashto[dll]&{}\undertag{}{e}{8pt}&
\overtag{\Circ}{a_{ww}}{8pt}\dashto[urr]\dashto[drr]&&\Vdots\\
&&&&&&&\\
&&&&&&{\hbox to 0pt{\hss$\underbrace{\hbox to 70pt{}}$\hss}}&\\
&&&&&&{\Delta(M)_{ve}}&}$$


\begin{prop}
Let $v$ be a node in $\Gamma(M)$. Then the sign $\epsilon$ at $v$ is
$\epsilon=-\sign(a_{vv})$, where $a_{vv}$ is the entry of 
$A(M)\inv$ corresponding to the node $v$. 
\end{prop}

Let $\{e_v\}\subset\E^*=\hom(\E,\Z)\subset\E\otimes \Q$ and
$\{\tilde{e}_v\}\subset\overline{\E}^*= \hom(\overline{\E},\Z)
\subset\overline{\E}\otimes\Q$ be the dual bases.  

Define the \emph{discriminant group} as the finite abelian group
\begin{align*}
D(\Delta(M)):=\E^*/\E,
\end{align*}
the order of $D(\Delta(M))$ is $\det(M)\colon =\num{\det(-A(M))}$. The
intersection pairing of $\Delta(M)$ induces pairings of $\E\otimes\Q$
into $\Q$ and $D(M)$ into $\Q/\Z$. We then get the following facts
about discriminant groups from Section 5 of \cite{neumannandwahl2}.

\begin{prop}
Consider a collection $e_w$, where $w$ runs over all leaves of
$\Gamma(M)$. Then $D(M)$ is generated by the images of these $e_w$.
\end{prop}

\begin{prop}
Let $e_1,\dots,e_n$ be the elements of the dual basis of $\E^*$
corresponding to the leaves of $\Gamma(M)$. Then the homomorphism
$\E^*\to \Q^n$ defined by 
\begin{align*}
e\mapsto (e\cdot e_1,\dots,e\cdot e_n)
\end{align*}
induces an injection
\begin{align*}
D(M)\hookrightarrow (\Q/\Z)^n.
\end{align*}
In fact, each non-trivial element of $D(M)$ gives an element of
$(\Q/\Z)^n$ with at least two non-zero entries.
\end{prop}

We will embed $(\Q/\Z)^n$ into $\C^n$ via the following map
\begin{align*}
(\dots,r,\dots)\mapsto(\dots,\exp(2\pi ir),\dots)=:[\dots,r,\dots].
\end{align*}

\begin{prop}
Let $w_1,\dots,w_n$ be the leaves of $\Gamma(M)$, then the
discriminant group $D(M)$ is naturally represented by a diagonal
action on  $\C^n$, where the entries are $n$-tuples of
$\num{\det(M)}$'th roots of unity. Each leaf $w_j$ corresponds to an
element 
\begin{align*}
\big[e_{w_j}\cdot e_{w_1},\dots,e_{w_j}\cdot e_{w_n}\big] :=\big(\exp(2\pi
ie_{w_j}\cdot e_{w_1}),\dots,(\exp(2\pi ie_{w_j}\cdot e_{w_n})\big), 
\end{align*}
and any $n-1$ of these generate $D(M)$. The representation contains no
pseudoreflections, i.e.\ non-identity elements fixing a hyperplane.
\end{prop}

\section{Congruence Condition for Graph
  Orbifolds}\label{orbifoldcongruencecondition} 

The plumbing diagram $\Delta$ we use will be assumed to be
\emph{quasi-minimal}, this means that all weights on strings of
$\Delta$ have weights less that $-1$, unless the string consist of a
single vertex with weight $-1$. 

To any string
$$
\xymatrix@R=6pt@C=24pt@M=0pt@W=0pt@H=0pt{
\overtag{\Circ}{-b_1}{8pt}\lineto[r]&
\overtag{\Circ}{-b_2}{8pt}\dashto[rr]&& \overtag{\Circ}{-b_k}{8pt}
\hbox to 0pt {\hss}}
$$
in $\Delta$ one associates a continued fraction
\begin{align*}
n/p:=b_1-\cfrac{1}{b_{2}-\cfrac{1}{b_{3}-\dots}}.
\end{align*}
We associate $1/0$ to the empty string. We need the following standard
facts about this relation ship which proofs are not hard and can be
found many places.
\begin{lemma}
Reversing a string with continued fraction $n/p$ gives one with
continued fraction $n/p'$ with $pp'\equiv 1(\mod n)$. Moreover the
following relations holds:
$$
\xymatrix@R=6pt@C=24pt@M=0pt@W=0pt@H=0pt{
n=\det\big(\hspace{.25cm} \overtag{\Circ}{-b_1}{8pt}\lineto[r]&
\overtag{\Circ}{-b_2}{8pt}\dashto[rr]&&
\overtag{\Circ}{-b_k}{8pt}\hspace{.25cm} \big)\\
p=\det\big(\hspace{.25cm} \overtag{\Circ}{-b_2}{8pt}\lineto[r]&
\overtag{\Circ}{-b_3}{8pt}\dashto[rr]&&
\overtag{\Circ}{-b_k}{8pt}\hspace{.25cm} \big)\\
p'=\det\big(\hspace{.25cm} \overtag{\Circ}{-b_1}{8pt}\lineto[r]&
\overtag{\Circ}{-b_2}{8pt}\dashto[rr]&&
\overtag{\Circ}{-b_{k-1}}{8pt}\hspace{.25cm} \big)
\hbox to 0pt {~,\hss}}
$$
and the continued fraction in the last case is $p'/n'$ with
$n'=(pp'-1)/n$. 

For each $n/p\in[1,\infty]$, there is a unique quasi minimal string,
and in this case the continued fraction associated to the
reverse direction is $n/p'$ where $p'$ is the unique number satisfying
$p'\leq n$ and $pp'\equiv 1(\mod n)$.
\end{lemma}

As we saw in last section the discriminant group $D(M)$ acts
diagonally on $\C^n$. Viewing the $z_{w_i}$'s as linear functions on
$\C^n$, $D(M)$ act naturally on $\C[z_{w_1},\dots,z_{w_n}]$, where $e$
acts on a monomial as
\begin{align*}
\prod z_{w_i}^{\alpha_{w_i}}\mapsto\Big[-\sum(e\cdot
e_{w_i})\Big]\prod z_{w_i}^{\alpha_{w_i}}. 
\end{align*}
This is the same as saying that the group transforms each monomial
according to the character
\begin{align*}
e\mapsto\exp\Big(-2\pi i\sum(e\cdot e_{w_j})\alpha_{w_j}\Big).
\end{align*}

Now we will return the setting of $\Gamma(M)$ satisfying the
semigroup conditions, so we have the notion of admissible monomial.

\begin{defn} Let $M$ be a graph orbifold with splice diagram
  $\Gamma(M)$ satisfy the semigroup condition. Let $\Delta(M)$ be
  a plumbing diagram of $M$, then $\Delta(M)$ satisfy the \emph{orbifold
  congruence condition} if for each node $v$ of $\Gamma(M)$ and
adjacent edge $e$, one can choose admissible monomials $M_{ve}$ so that
$D(M)$ transforms these monomials according to the same character.
\end{defn}

Notice that if $M$ is a manifold, then this definition is the same as
the definition of congruence condition (Definition 6.3) in
\cite{neumannandwahl2}. Next we write down explicit equation in terms
of $\Gamma(M)$ and $\Delta(M)$ for the congruence condition.

\begin{lemma}\label{dotproduct}
The matrix $(e_v\cdot e_{v'})$ where $v,v'\in\vx(\Delta(M)) $ is the
inverse matrix of $A(M)$, and the matrix $(\overline{e}_v\cdot
\overline{e}_{v'})$ is the inverse matrix of $A(\overline{M})$
\end{lemma}

\begin{proof}
This follows from elementary linear algebra.
\end{proof}

\begin{lemma}\label{dotproductcomparison}
For any $v,v'\in\vx(\Delta(M)) $, we have that
\begin{align}
e_v\cdot e_{v'} &=\frac{\overline{e}_v\cdot \overline{e}_{v'}}{o_{v'}}.
\end{align}
\end{lemma}

\begin{proof}
This follows from Cramer's rule, i.e. if $M_{ij}$ is the $i,j$'th minor
of $A(M)$ and $\overline{M}_{ij}$ is the $i,j$'th minor
of $A(\overline{M})$, then $M_{ij}$ is $\overline{M}_{ij}$ with each
column $l$ multiplied by the corresponding orbifold degree $o_l$,
hence $\det(M_{ij})=\big(\prod_{l\neq j}o_l\big)\det(\overline{M}_{ij})$ and
\begin{align*}
A(M)_{i,j}\inv&=\frac{1}{\det(A(M))}(-1)^{i+j}\det(M_{ij})\\ &=
\frac{\prod_{l\neq j}o_l}{\prod_{l}o_l}
\frac{1}{\det(A(\overline{M}))}(-1)^{i+j}\det(\overline{M}_{ij})   
= \frac{1}{o_j}A(\overline{M})_{i,j}\inv.
\end{align*}
\end{proof}

\begin{lemma}
If $v$ and $v'$ are different vertices in $\Delta(M)$, corresponding
to leaves of $\Gamma(M)$, then
\begin{align}
e_v\cdot e_{v'}=-\frac{o_vl_{vv'}}{d}.
\end{align}
\end{lemma}

\begin{proof}
$\overline{e}_v\cdot \overline{e}_{v'}=-\frac{\tilde{l}_{vv'}}{\tilde{d}}$
by Lemma 6.4 in \cite{neumannandwahl2}, 
 but by the
definition of $\overline{M}$ it follows that
$l_{vv'}=\overline{l}_{vv'}\prod_{v''\neq v,v'}o_{v''}$ and
$d=\overline{d}\prod_{v''} o_{v''}$, so using lemma \ref{dotproduct} we
get
\begin{align*}
e_v\cdot e_{v'} &= \frac{\overline{e}_v\cdot \overline{e}_{v'}}{o_{v'}}=
-\frac{1}{o_{v'}}\frac{\overline{l}_{vv'}}{\overline{d}}
=-\frac{1}{o_{v'}}\frac{l_{vv'}}{\prod_{v''\neq
    v,v'}o_{v''}}\frac{\prod_{v''}o_{v''}}{d}=-\frac{o_vl_{vv'}}{d}.
\end{align*}
\end{proof}     

\begin{prop}
If $v$ is a node in $\Gamma(M)$ and $w$ is a adjacent leaf and the
continued fraction associated to the string in $\Delta(M)$ (and
$\Delta(\widetilde{M})$) is given by $(n_w/o_w)/p$, where $n_w$ is the
weight of the leaf and $o_w$ is the orbifold degree of the leaf. Let
$p'$ be the smallest positive integer such that $p'p=1\mod
(n_w/o_w)$. Let $\{n_i\}_{i=1}^k$ be the other weights adjacent to $v$
and let $N=\prod_{i=1}^kn_i$. Then
\begin{align}
e_w\cdot e_w &=-\frac{o_wN}{dn_w}-\frac{p'}{n_w}.
\end{align}
\end{prop}

\begin{proof}
By using \ref{dotproduct} and the formula for $\overline{e}_w\cdot
\overline{e}_w$ given by proposition 6.6 in \cite{neumannandwahl2} we get
\begin{align*}
e_w\cdot e_w &=\frac{\overline{e}_w\cdot \overline{e}_w}{o_w}
=-\frac{1}{o_w}\Big(
\frac{n_w/o_w\prod_{i=1}^kn_i/o_i}{(n_w/o_w)^2d/o_i\prod_i o_i}+
\frac{p'}{n_w/o_w}\Big)=-\frac{o_wn_w\prod_in_i}{n_w^2d}-\frac{p'}{n_w}.
\end{align*}
\end{proof}

\begin{cor}
The class of $e_{w'}$, where $w'$ is a leaf, transforms the monomial
$\prod z_w^{\alpha_w}$ by multiplication by the root of unity 
\begin{align}
\Big[\sum_{w\neq
  w'}\alpha_w\frac{o_{w'}l_{ww'}}{\det(\Gamma)}-\alpha_{w'}e_{w'}\cdot
e_{w'}\Big].
\end{align}
\end{cor}

We are now able to give formulas for checking the congruence
condition.

\begin{prop}
Let $\Delta$ be an orbifold plumbing diagram which splice diagram
$\Gamma$ satisfy the semigroup condition. Then the orbifold congruence
condition is
equivalent to the following: for every node $v$ and adjacent edge $e$,
there is an admissible monomial $M_{ve}=\prod z_w^{\alpha_w}$ where
$w$ is a leaf in $\Gamma_{ve}$, so that for every leaf $w'$ of
$\Gamma_{ve}$,
\begin{align}
\Big[\sum_{w\neq
  w'}\alpha_w\frac{o_{w'}l_{ww'}}{\det(\Gamma)}-\alpha_{w'}e_{w'}\cdot
e_{w'}\Big]=\Big[\frac{o_{w'}l_{vw'}}{\det(\Gamma)}\Big].
\end{align}
\end{prop} 

\begin{proof}
The proof follows exactly as the prof given for the manifold case in
Proposition 6.8 in \cite{neumannandwahl2}. 
\end{proof}

We will now look closer to how the congruence condition looks in the
two node case. Let $\Gamma$ be the splice diagram
$$\splicediag{8}{30}{
  &&&&&\\
  &\overtag\Circ {w_{01}} {8pt}&&&&\overtag\Circ {w_{11}} {8pt}\\
  \Gamma=&\Vdots&\overtag\Circ {v_0} {8pt}\lineto[ul]_(.5){n_{01}}
  \lineto[dl]^(.5){n_{0k_0}}
  \lineto[rr]^(.25){r_0}^(.75){r_1}&& \overtag\Circ{v_1}{8pt}
  \lineto[ur]^(.5){n_{11}}
  \lineto[dr]_(.5){n_{1k_1}}&\Vdots\\
  &\undertag\Circ {w_{0k_0}} {4pt}&&&&\undertag\Circ {w_{1k_1}} {4pt}\\
&&&&&\hbox to 0 pt{~.\hss} }$$
and let the orbifold degree corresponding to $n_{ij}$ be $o_{ij}$. let
$\Delta$ be a plumbing diagram where we have made sure by blowing up
that all the vertices $v$ with $o_v\neq 1$ have valence one. Then
$\Delta$ is going to look like 

$$
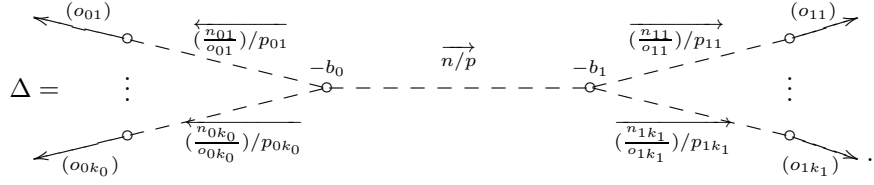
\xymatrix@R=6pt@C=24pt@M=0pt@W=0pt@H=0pt{
&&&&&&&&&&&&\\
&\Circ\ar[ul]_{(o_{01})} &&&&&&&&&&\Circ \ar[ur]^{(o_{11})}&\\
\Delta=&\Vdots&&&\overtag{\Circ}{-b_0}{8pt}
\dashto[rrrr]\dashto[ulll]_{\overleftarrow{(\frac{n_{01}}{o_{01}})/p_{01}}}
\dashto[dlll]^{\overleftarrow{(\frac{n_{0k_0}}{o_{0k_0}})/p_{0k_0}}}
&&{}\overtag{}{\overrightarrow{n/p}}{16pt}&& 
\overtag{\Circ}{-b_1}{8pt}
\dashto[urrr]^{\overrightarrow{(\frac{n_{11}}{o_{11}})/p_{11}}} 
\dashto[drrr]_{\overrightarrow{(\frac{n_{1k_1}}{o_{1k_1}})/p_{1k_1}}}
&&&\Vdots&\\
&\Circ \ar[dl]^{(o_{0k_0})}&&&&&&&&&&\Circ \ar[dr]_{(o_{1k_1})}&\\
&&&&&&&&&&&&\hbox to 0pt {~.\hss}}
$$
Let $p_{ij}'$ be the unique positive integer satisfying
$p_{ij}p_{ij}'\equiv 1 \mod (n_{ij}/o_{ij})$ and $p_{ij}'<
(n_{ij}/o_{ij})$. Let $N_i=\prod_jn_{ij}$. Then
the equations for the congruence condition becomes
\begin{align*}
\Big[\frac{o_{lj}N_0N_1}{n_{lj}d}\Big]&=
\Big[\sum_{\substack{i=1 \\i\neq
  j}}^{k_l}\alpha_{w_{li}}\frac{o_{lj}N_lr_l}{n_{li}n_{lj}d}
-\alpha_{lj}e_{w_{lj}}\cdot e_{w_{lj}}\Big]\\
&=\Big[\sum_{\substack{i=1 \\i\neq
  j}}^{k_l}\alpha_{w_{li}}\frac{o_{lj}N_lr_l}{n_{li}n_{lj}d}-
\alpha_{lj}\big(-\frac{o_{lj}N_lr_l}{n_{lj}n_{lj}}-\frac{p'_{lj}}{n_{lj}}\big)
  \Big]\\ 
 &=\Big[\frac{o_{lj}r_0r_1}{n_{lj}d}+\frac{p'_{lj}}{n_{lj}}
  \Big] 
 \end{align*}
Where we use that
$r_{1-l}=\sum_{i=1}^{k_l}\alpha_{li}\frac{N_l}{n_{li}}$ be the choice
of admissible monomials. This equality is of course equivalent to 
\begin{align*}
\Big[0\Big] &= \Big[\frac{o_{lj}(r_0r_1-N_0N_1)}{n_{lj}d}+\frac{p'_{lj}}{n_{lj}}
  \Big] =  \Big[\frac{o_{lj}n}{n_{lj}}+\frac{p'_{lj}}{n_{lj}}
  \Big]
\end{align*}
where we use that $nd=r_0r_1-N_0N_1$ by the edge determinant
equation (Corollary 3.3 in \cite{myarticle}. Since $n_{lj}/o_{lj}$ is an integer, the equation becomes
equivalent to the following
\begin{align*}
\frac{\alpha_{lj}}{o_{lj}}p_{lj}'\equiv -n \mod (n_{lj}/o_{lj}).
\end{align*}
Using the definition of $p_{ij}'$ gives us the following set of equations
one need to check for the congruence to be satisfied
\begin{align}
\frac{\alpha_{lj}}{o_{lj}}\equiv -np_{lj} \mod
(n_{lj}/o_{lj}).\label{congruenceeq} 
\end{align} 

\section{Splice diagram equations determining the universal abelian
  cover}\label{universalabcoveroforbifolds} 

Let $M$ be a graph orbifold with splice diagram $\Gamma=\Gamma(M)$
satisfying the semigroup condition. Assume that $\{\alpha_w\}$ is a
choice of semigroup coefficients such that $\Delta(M)$ satisfying the
orbifold congruence condition.

\begin{defn} 
The set of semigroup coefficients $\{\alpha_w\}$ is said to be
\emph{$M$ reducible} if the orbifold degree $o_w$ divides $\alpha_w$ for
each leaf $w\in\Gamma$.
\end{defn}

Let $\overline{M}$ be the underlying manifold of $M$, then the splice
diagram $\overline{\Gamma}=\Gamma(\overline{M})$ is equal to $\Gamma$
except that given an edge $e\in \Gamma$ at a node $v$, then the
edge weight $\overline{d}_{ve}$ in $\overline{\Gamma}$ is given by
$\overline{d}_{ve}=\frac{d_{ve}}{\prod_{w\in\Gamma_{ve}}o_w}$ where
$d_{ve}$ is the edge weight in $\Gamma$. If $\{\alpha_w\}$ is a set on
$M$ reducible semigroup coefficients then $\{\overline{\alpha}_w\}$,
where $\overline{\alpha}_w=\alpha_w/o_w$, is a
set of semigroup coefficients for $\overline{M}$ since
\begin{align*}
\overline{d}_{ve}&=\frac{d_{ve}}{\prod_{w\in\Gamma_{ve}}o_w}=
\frac{\sum_{w\in\Gamma_{ve}}\alpha_wl_{vw}'}{\prod_{w\in\Gamma_{ve}}o_w}=
\sum_{w\in\Gamma_{ve}}\alpha_w/o_w\frac{l_{vw}'}{
  \prod_{\substack{w'\in\Gamma_{ve} \\ w'\neq w}}o_{w'}}= 
\sum_{w\in\overline{\Gamma}_{ve}}\overline{\alpha}_w\overline{l}_{vw}' 
\end{align*}

\begin{prop}
Assume that $\{\alpha_w\}$ satisfy the orbifold congruence condition
for $\Delta(M)$ then $\{\overline{\alpha}_w\}$ satisfy the congruence
condition for $\Delta(\overline{M})$.
\end{prop}

\begin{proof}
Let $v\in\overline{\Gamma}$ be a node and $w'\in\overline{\Gamma}_{ve}$
be a leaf, then
\begin{align*}
&\Big[\sum_{w\neq
  w'}\overline{\alpha}_w\frac{\overline{l}_{ww'}}{\det(\overline{M})}
-\overline{\alpha}_{w'}\overline{e}_{w'}\cdot\overline{e}_{w'}\Big]\\ &= 
\Big[\sum_{w\neq
  w'}\alpha_w/o_w\frac{l_{ww'}}{\prod_{v'\neq w,w'}o_{v'}}
\frac{\prod_{v'}o_{v'}}{\det(M)} 
-\frac{\alpha_{w'}}{o_{w'}}o_{w'}e_{w'}\cdot e_{w'}\Big]\\ &=
\Big[\sum_{w\neq
  w'}o_{w'}\alpha_w\frac{l_{ww'}}{\det(M)} 
-\alpha_{w'}e_{w'}\cdot e_{w'}\Big]\\
&=\Big[\frac{o_{w'}l_{vw'}}{\det(M)}\Big]=
\Big[\frac{o_{w'}\overline{l}_{vw'}\prod_{v'\neq w'}o_{v'}}{\det(M)
  \prod_{v'}o_{v'}}\Big]=\Big[\frac{\overline{l}_{vw'}}{\det(M)}\Big]. 
\end{align*}
Where we use that $\{\alpha_w\}$ satisfy the orbifold congruence
condition to get from the third line to the fourth.
\end{proof}

Associate to each leaf $w\in\Gamma$ a variable $z_w$, for each node
$v\in\Gamma$ let $\{\alpha_{vw}\}$ be a reducible choice of semigroup
coefficients satisfying the orbifold congruence condition, and let
$\overline{\alpha}_{vw}=\frac{\alpha_{vw}}{o_w}$. Let
$Z_{ve}=\prod_{w\in\Gamma_{ve}}z_w^{\alpha_{vw}}$ and
  $\overline{Z}_{ve}=\prod_{w\in\Gamma_{ve}}z_w^{\overline{\alpha}_{vw}}$. Let
    $t$ be the number of leaves of $\Gamma$ and Let
    $V$ be the subvariety of $\C^t$ defined by the equations
\begin{align*}
\Sigma_e a_{vie}Z_{ve}=0,\ v\ \text{a node,}\ i=1,\dots,\delta_v-2,
\end{align*}
where for all $v$ the maximal minors of the
$((\delta_v-2)\times\delta_v)$-matrix $(a_{vie})$ have maximal
rank. Likewise let $\overline{V}$ be the subvariety of $\C^t$ defined
by the equations 
\begin{align*}
\Sigma_e a_{vie}\overline{Z}_{ve}=0,\ v\ \text{a node,}\ i=1,\dots,\delta_v-2,
\end{align*} 
with the same choice of $a_{vie}$ as for $V$.
Define the map $\morf{F}{\C^t}{\C^t}$ by
\begin{align}
F(z_{w_1},\dots,z_{w_t})=(z_{w_1}^{o_{w_1}},\dots,z_{w_t}^{o_{w_t}}). 
\end{align} 
Then $F(V)=\overline{V}$, and $F$ is a branched abelian cover of
$\C^t$ with $\deg(F)=\prod_wo_w$, branched over $B=\bigcup_w\{z_w=0\ \vert\
o_w>1\}$. Now $\morf{F\vert_V}{V}{\overline{V}}$ is a branched abelian
cover, branched over $\overline{V}\bigcap B$.

Let $X$ be a singularity which has resolution $\Delta(\overline{M})$
and hence has link $\overline{M}$. Since the equations for
$\overline{V}$ satisfy the congruence conditions for
$\Delta(\overline{M})$, $\overline{V}$ defines the
universal abelian cover of $X$ branched over the origin by the work of
Neumann and Wahl \cite{neumannandwahl2}. Let
$\morf{\overline{\pi}}{\overline{V}}{X}$ denote the covering map, then
    $\deg(\overline{\pi})=\num{H_1(\overline{M})}
    =\frac{\num{H_1^{orb}(M)}}{\prod_w o_w}$. 

Now $\overline{M}$ embeds into $X$, so choose a small enough embedding
$i$ and let $L(\overline{V})=\overline{\pi}\inv(i(\overline{M}))$, then
$L(\overline{V})$ is homeomorphic to the link of $\overline{V}$. Let
$L(V)=F\inv(L(\overline{V}))$, then by choosing small enough
embedding $L(V)$ is homeomorphic to the link of $V$. Then the
restrictions of the maps $\morf{F\vert_{L(v)}}{L(V)}{L(\overline{V})}$
and
$\morf{i\inv\circ
  \overline{\pi}\vert_{L(\overline{V})}}{L(\overline{V})}{\overline{M}}$
are abelian covers, the first branched over $L(\overline{V})\bigcap
B$.

Let $\morf{f}{\overline{M}}{M}$ be the homeomorphism which identifies
$\overline{M}$ and $M$ as topological spaces. 

\begin{defn}  
Let $\morf{\pi}{L(V)}{M}$ be defined as $\pi=f\circ
i\inv\circ\overline{\pi}\vert_{L(\overline{V})}\circ F\vert_{L(V)}$.
\end{defn}
A priori $\pi$ is just a continues map, we next turn to prove that
$\pi$ is an abelian orbifold cover.
 
Let $K_w\subset M$ be the singular fiber corresponding to the leaf
$w\in\Gamma$. Let $S$ be the singular set of $M$, then
$S=\bigcup_{w:o_w>1}K_w$. We want to determine $\pi\inv(S)$. First
$f\inv(S)=S$ where we see $S$ as a subspace of $\overline{M}$, i.e.\
$S$ is the union of singular fibers of $\overline{M}$ for which
$o_w>1$, the last information is not detectable from
$\overline{M}$. Now $\overline{\pi}\vert_{L(\overline{V})}\inv(K_w)
=L(\overline{V})\bigcap\{z_w=0\}$ and hence 
\begin{align*}
(f\circ
i\inv\circ\overline{\pi}\vert_{L(\overline{V})})\inv(S)
=L(\overline{V})\bigcap\bigcup_{w:o_w>1}\{z_w=0\}=L(\overline{V})\bigcap
B 
\end{align*}
Now $\morf{f}{(\overline{M}-f\inv(S))}{(M-S)}$ is an abelian cover, since
 $f$ is the identity away from $S$,
 $\morf{\overline{\pi}\vert_{L(\overline{M})}}{(L(\overline{M})
   -B)}{(\overline{M}-f\inv(S))}$ is an abelian cover, since it is the
 restriction of an abelian cover to a union of fibers. And
 $\morf{F\vert_{L(V)}}{(L(V)-F\inv(B))}{(L(\overline{M})-B)}$ is an
 abelian cover, since it is the restriction of a branched abelian
 cover to the compliment of the branched locus. Hence
\begin{align*}
\morf{\pi\vert_{(L(V)-\pi\inv(S))}}{(L(V)-\pi\inv(S))}{(M-S)}
\end{align*}
is an abelian cover of degree $\deg(F)\deg(\overline{\pi})\deg(f)=(\prod_wo_w)
\frac{\num{H_1^{orb}(M)}}{\prod_w o_w}1=\num{H_1^{orb}(M)}$.

So next we turn to what happens in the neighborhood of an orbifold
curve.
 
\begin{prop}
Let $K_w\subset M$ be a orbifold curve of degree $o_w$ and let $N_{K_w}$ be a
solid torus neighborhood of $K_w$, then
$\morf{\pi\vert_{\pi\inv(N_{K_w})}}{\pi\inv(N_{K_w})}{N_{K_w}}$ is an
abelian orbifold cover of degree $\num{H_1^{orb}(M)}$.
\end{prop}

\begin{proof}
We both need to show that the exist open set $U\subset \R^3$ and
$D\subset\R^2$ and a branched abelian cover $\morf{\tilde{\pi}}{U}{D\times
  S^1}$ and a homomorphism $\morf{\psi}{\Z/a\Z}{\Z/\alpha\Z}$ such
that $\pi$ is equivariant with respect to $\psi$ and the following
diagram commutes 
\begin{align*}
\xymatrix@=5mm{
U\ar[rrr]^{\tilde{\pi}}\ar[d] &&& D\times S^1\ar[d]\\
U/(\Z/a\Z)\ar[d] &&& (D\times S^1)/(\Z/\alpha\Z)\ar[d]\\
\pi\inv(N_{K_w})\ar[rrr]^{\pi}&&& N_{K_w}.
}
\end{align*}
The vertical maps are the ones
given from the orbifold 
structures of $\pi\inv(N_{K_w})$ and $N_{K_w}$, i.e.\ the upper maps
are quotient maps and lower maps homeomorphisms. We can choose $D$ to
be a disk, and by choosing it small enough (i.e.\ choosing $N_{K_w}$
small enough) we get that $\pi\inv(N_{K_w})$ is a disjoint union of
solid torus neighborhoods $V_K$ of an singular fiber $K$, hence we
can choose $U$ to be a disjoint unions of $D\times S^1$. We now only
need to see that the diagrams commute for each of these component. 

Let $t_n=e^{2\pi i/n}$, then $\Z/\alpha\Z$ acts on $(D\times S^1)$ by
$(x,s)\to(t_\alpha^{v_{\alpha}} x,t_\alpha^{o_w} s)$ where
$\gcd(\alpha,v_\alpha)=1$.  
Likewise $\Z/a\Z$ acts on $(D\times S^1)$ by
$(x,s)\to(t_a^{v_a} x,t_a s)$. 

Let $N_k'=f\inv(N_{K_w})\subset \overline{M}$, and then 
$\overline{\pi}\inv(N_K')=\bigcup_{i=1}^m\overline{V}_K'$, where
$\overline{V}_K'$ is a solid torus neighborhood of a singular fiber
of degree $a'$. Let $\alpha'=\alpha/o_w$ and $\beta'=\beta/o_w$ then
the orbifold structure on $N_K'$ is given by $\Z/\alpha'\Z$ acting on
$D\times S^1$ by $(t_{\alpha'}^{v_{\alpha'}}x,t_{\alpha'}s)$, where
$v_{\alpha'}\beta'\equiv -1\mod \alpha'$ and this implies that
$v_\alpha\equiv v_{\alpha'}\mod\alpha'$. Since
$\overline{\pi}$ is an abelian orbifold cover there exist an
homomorphism $\morf{\psi'}{\Z/a'\Z}{\Z/\alpha'\Z}$, and an abelian
cover $\overline{\pi}'$ such that the following diagram commutes
\begin{align*}
\xymatrix@=5mm{
D\times S^1\ar[rrr]^{\overline{\pi}'}\ar[d] &&& D\times S^1\ar[d]\\
(D\times S^1)/(\Z/a'\Z)\ar[d] &&& (D\times S^1)/(\Z/\alpha'\Z)\ar[d]\\
\overline{V}_K'\ar[rrr]^{\overline{\pi}}&&& N_K',
} 
\end{align*}
and
$\psi(t_{a'})\overline{\pi}'(x,s)=\overline{\pi}'(t_{a'}^{v_{a'}}x,t_{a'}s)$. 


Now since $\overline{\pi}$ is an abelian cover, it follows that
$\alpha'=\lambda a'$ for some $\lambda$ and hence,
$\psi'(t_a')=t_{\alpha'}^\lambda$.

Now restricting $F$ to $V_K$ one sees that $F$ is a branched
cover and if $V_K=D\times S^1$ then
$F(x,s)=(x^{o_w},s)\in\overline{V}_K'=D\times S^1$. Now the Seifert
fibered structure on $\overline{V}_K'$ is define by the curve of slope
$b'/a'$, and it lifts to the curve of slope $o_wb'/a'$ and hence
$a=a'$ and $b=o_wb'$ and hence $v_{a'}\equiv o_wv_a\mod a'$.

We can now define $\tilde{\pi}$ and
$\psi$. $\tilde{\pi}(x,s)=\overline{\pi}'(x^{o_w},s)$ and
$\psi(t_a)=t_\alpha^{o_w\lambda}$. So first we need to check that
$\tilde{\pi}$ is equivariant with respect to $\psi$
\begin{align*}
 \psi(t_a)\widetilde{\pi}(x,s)&=
 t_\alpha^{o_w\lambda}\overline{\pi}'(x^{o_w},s) 
 =\psi'(t_a')\overline{\pi}'(x^{o_w},s)=
\overline{\pi}'(t_{a'}^{v_{a'}}x^{o_w},t_{a'} s)\\
&=\overline{\pi}'(t_{a'}^{o_wv_a}x^{o_w},t_{a'} s)=
\overline{\pi}'((t_a^{v_a}x)^{o_w},t_{a'} s)=
\tilde{\pi}(t_a^{v_a}x,t_as)  
\end{align*}

So what is left is just checking that the diagram commutes. Start by
taking $(x,s)\in D\times S^1$, then the one composition is taking
$\tilde{\pi}(x,s)$ and send it to the class in $(D\times
S^1)/(\Z/\alpha\Z)$, and the other composition is sending $(x,s)$ into
the class in  $(D\times S^1)/(\Z/a\Z)$ and then take $\pi$. If we
denote the class in  $(D\times S^1)/(\Z/k\Z)$ where the action is
given by the integers $k,l$ by $[x,s]_{(k,l)}$, then we need to see
that $\pi([x,s]_{(a,b)})= [\tilde{\pi}(x,s)]_{(\alpha,\beta)}$. Now
$\pi([x,s]_{(a,b)})=f(\overline{\pi}(F([x,s]_{(a,b)})))$ by
definition. Since the Seifert fibered structure on $V_K$ is given by
pulling back the Seifert fibered structure on $\overline{V}_K'$ by
$(x^{o_w},s)$, $F([x,s]_{(a,b)})=[x^{o_w},s]_{(a',b')}$. By construction
we have that
$\overline{\pi}([x,s]_{(a',b')})=[\overline{\pi}'(x,s)]_{(\alpha',\beta')}$, 
and by definition
$f([x,s]_{(\alpha',\beta')})=[x,s]_{(\alpha,\beta)}$, hence
$\pi([x,s]_{(a,b)})=[\overline{\pi}''(x^{o_w},s)]_{(\alpha,\beta)}$,
and the diagram commutes. 

Last we need to calculate the degree of $\pi\vert_U$. First the degree of
$\morf{\pi\vert_{V_K}}{V_K}{N_{K_w}}$ is $o_w\lambda$. The degree of
$\overline{\pi}$ on
$\overline{\pi}\inv(N_K')=\bigcup_{i=1}^m\overline{V}_K'$ is
$\num{H_1(\overline{M})}$, hence
$m=\num{H_1(\overline{M})}/\lambda$. $U=\bigcup_{i=1}^{(\prod_{w'\neq
    w}o_{w'})}\overline{\pi}\inv(\overline{V}_k')$ and hence the number of
component of $U$ is $(\num{H_1(\overline{M})}/\lambda)\prod_{w'\neq
    w}o_{w'}$ and $\deg(\pi\vert_U)=
  o_w\lambda(\num{H_1(\overline{M})}/\lambda)\prod_{w'\neq w}o_{w'}
  =\num{H_1(\overline{M})}\prod_{w'}o_{w'}=\num{H_1^{orb}(M)}$.
\end{proof}

Combining the above results gives us the following theorem

\begin{thm}\label{linkascover}
Let $M$ be a rational homology sphere graph orbifold with splice
diagram $\Gamma$ satisfying the semigroup condition. Suppose there
exist a graph orbifold $M'$ also with splice diagram $\Gamma$, and a
set of reducible semigroup coefficients $\{\alpha\}$ for $M'$ satisfying
the orbifold congruence condition. Then the link of the complete
intersection defined by $(\Gamma,\{\alpha\})$ is homeomorphic to the
universal abelian cover of $M$.
\end{thm}

\begin{proof}
The above show that $\morf{\pi}{L(V)}{M'}$ is an orbifold abelian
cover of degree $\num{H_1^{orb}(M)}$, and hence the universal abelian
cover of, $M'$, combining this with the second main theorem of
\cite{myarticle} gives the result.
\end{proof}

So to prove that the the splice diagram always define the universal
abelian cover, one just have to show that given $M$ with $\Gamma(M)$
satisfying the semigroup condition, there always exist a $M'$ with a
reducible set of admissible monomials satisfying the orbifold
congruence condition such that $\Gamma(M')=\Gamma(M)$. We will in the next
section show this is always true in the case of a splice diagram with
only two nodes, by constructing such a $M'$ from any two node splice
diagram satisfying the semigroup condition.

\section{Algorithm for construction an orbifold with a given two node
  splice diagram}\label{constructingorbifolds}

In this section we will make an algorithm which given any two node
splice diagram $\Gamma$ satisfying the ideal generator condition gives a graph
orbifold $M$ with $\Gamma(M)=\Gamma$. We will construct $M$ by giving
a plumbing diagram $\Delta$ such that $\Delta=\Delta(M)$. We will not
give a complete plumbing diagram, but specify the orbifold degrees,
the weight at the nodes, and the continued fraction associated to the
strings. From this data one can the obtain the complete plumbing
diagram if needed. Let the splice diagram look like the following
$$\splicediag{8}{30}{
  &&&&&\\
  &\Circ &&&&\Circ \\
  \Gamma=&\Vdots&\overtag\Circ {\epsilon_0} {8pt}\lineto[ul]_(.5){n_{01}}
  \lineto[dl]^(.5){n_{0k_0}}
  \lineto[rr]^(.25){r_0}^(.75){r_1}&{}\undertag{}{e}{8pt}
 &\overtag\Circ{\epsilon_1}{8pt}
  \lineto[ur]^(.5){n_{11}}
  \lineto[dr]_(.5){n_{1k_1}}&\Vdots\\
  &\Circ &&&&\Circ\hbox to 0 pt{~.\hss} }$$
Let $N_j=\prod_{i}n_{ji}$, and let $D$ be the edge determinant of the
  central edge.
The plumbing diagram will be given by
$$
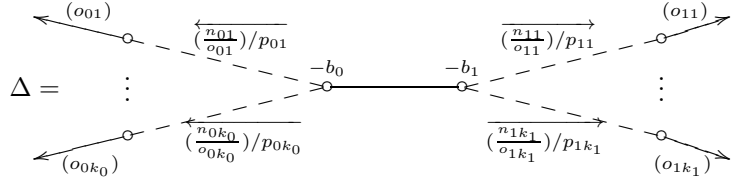
\xymatrix@R=6pt@C=24pt@M=0pt@W=0pt@H=0pt{
&&&&&&&&&&\\
&\Circ\ar[ul]_{(o_{01})} &&&&&&&&\Circ \ar[ur]^{(o_{11})}&\\
\Delta=&\Vdots&&&\overtag{\Circ}{-b_0}{8pt}
\lineto[rr]\dashto[ulll]_{\overleftarrow{(\frac{n_{01}}{o_{01}})/p_{01}}}
\dashto[dlll]^{\overleftarrow{(\frac{n_{0k_0}}{o_{0k_0}})/p_{0k_0}}}
&& 
\overtag{\Circ}{-b_1}{8pt}
\dashto[urrr]^{\overrightarrow{(\frac{n_{11}}{o_{11}})/p_{11}}} 
\dashto[drrr]_{\overrightarrow{(\frac{n_{1k_1}}{o_{1k_1}})/p_{1k_1}}}
&&&\Vdots&\\
&\Circ \ar[dl]^{(o_{0k_0})}&&&&&&&&\Circ \ar[dr]_{(o_{1k_1})}&\\
&&&&&&&&&&\hbox to 0pt {~.\hss}}
$$
So we need to specify $o_{ji}$, $p_{ji}$ and $b_j$, from the
information given by $\Gamma$.
\begin{itemize}
\item First chose integer $\alpha_{ji}$ such that $r_{1-j}=\sum_i
  \alpha_{ji}\frac{N_j}{n_{ji}}$, these integer exist since $\Gamma$
  satisfy the ideal generator condition. If $\Gamma$ furthermore
  satisfy the semigroup condition, then the $\alpha_{ji}$'s can be
  chosen to be non negative, and the choice of $\alpha_{ji}$'s is a
  choice of semigroup coefficients for $r_{1-j}$. 
\item Let $\lambda_{ji}$ be the smallest integer, such that
  $\lambda_{ji}n_{ji}\geq \epsilon_j\alpha_{ji}$ if $D>0$, and 
$\lambda_{ji}n_{ji}\geq -\epsilon_j\alpha_{ji}$ if $D<0$
\item Let $o_{ji}=\gcd(n_{ji},\alpha_{ji})$.
\item Let $p_{ji}=\frac{\lambda_{ji}n_{ji}-\alpha_{ji}}{o_{ji}}$ if
  $D>0$, and $p_{ji}=\frac{\lambda_{ji}n_{ji}+\alpha_{ji}}{o_{ji}}$ if
  $D<0$.
\item Let $b_j=\sum_i \lambda_{ji}$.
\end{itemize}
Notice that $\gcd(n_{ji}/o_{ji},p_{ji})=1$ so these choices gives a
well-defined plumbing.
\begin{prop}
Let $M$ be the graph orbifold given by $\Delta$ with the above
choices, then $\Gamma(M)=\Gamma$ and $\num{H_1^{orb}(M)}=D(e)$.
\end{prop}

\begin{proof}
Since the weight to the leaves in $\Gamma(M)$ is
$(n_{ji}/o_{ji})o_{ji}$, it is the same weight as in $\Gamma$. Next we
start by considering the case that $D>0$, then the unnormalized
edge determinant equation (Lemma 3.2 in \cite{myarticle}) implies that $\det(M)>0$, so the
only thing to check is that $\tilde{r}_j$, the weights associated to
the central string in $\Gamma(M)$, is $\epsilon_{1-j}r_j$.
\begin{align*}
\tilde{r}_{1-j} &=\big(\prod_io_{ji}\big)\det(\Delta(\widetilde{M})_{v_{j+1}e})
=\big(\prod_io_{ji}\big)\big(\prod_in_{ji}/o_{ji}\big)\big(
b_j-\sum_i\frac{p_{ji}}{n_{ji}{o_ji}}\big) \\
&= N_j\big(\sum_i\lambda_{ji}-\sum_ip_{ji}o_{ji}/n_{ji}\big)
=\sum_i\big(\lambda_{ji}N_j-\frac{\lambda_{ji}n_{ji}-
  \epsilon_j\alpha_{ji}}{o_{ji}}o_{ji}\frac{N_j}{n_{ji}}\big) \\
&=\sum_i\epsilon_j\alpha_{ji}\frac{N_j}{n_{ji}} =\epsilon_jr_{1-j}.
\end{align*}   
The case whit $D<0$ is similar, but now $\det(M)<0$ and
$\tilde{r}_j=-\epsilon_{1-j}r_j$.
The last statement follows from the edge determinant equation
(Corollary 3.3 in \cite{myarticle}), since
the fiber intersection number of $e$ is $1$, or by using the above
calculation to calculate $\det(\Delta)$.
\end{proof} 
Notice that $M$ do depend on the choice of $\alpha_{ji}$'s.
\begin{cor}
Let $\Gamma$ be a splice diagram satisfying the semigroup
condition. Then if $M$ is a graph orbifold given by the above
algorithm, then $M$ satisfy the orbifold congruence condition.
\end{cor}

\begin{proof}
We need to check that the equations $\frac{\alpha_{lj}}{o_{lj}}\equiv
-np_{ji} \mod (n_{ji}/o_{ji})$ given by \eqref{congruenceeq} are
satisfied. Now $n=1$ so by definition
\begin{align*}
p_{ji}=\frac{\lambda_{ji}n_{ji}-\alpha_{ji}}{o_{ji}}=
\lambda_{ji}\frac{n_{ji}}{o_{ji}}-\frac{\alpha_{ji}}{o_{ji}}.
\end{align*}
Which implies that $-\frac{\alpha_{lj}}{o_{lj}}\equiv
p_{ji} \mod (n_{ji}/o_{ji})$, and hence the congruence condition is satisfied. 
\end{proof}

Notice that the $\alpha_{ji}$'s are a reducible set of semigroup
coefficients by the definition of the $o_{ji}$'s. So combining this with
Theorem \ref{linkascover} we get that given a two node splice diagram $\Gamma$
satisfying the semigroup condition, the link of any splice diagram
surface singularity is homeomorphic to the universal abelian cover of
any graph orbifold with $\Gamma$ as its splice diagram.

Now this method for proving that the link of the splice diagram
equations are the universal abelian covers does not easily generalize
to more the two nodes. Already in the $3$-node case is the semigroup
(or ideal) condition not sufficient for a splice diagram to be
realized by a graph orbifold. The following diagram
$$\splicediag{8}{30}{ 
 &&&&\Circ &&&\\
 &\Circ &&& &&&\Circ \\
 \Gamma=& &\Circ \lineto[ul]_(.5){2}
  \lineto[dl]^(.5){3}
  \lineto[rr]^(.25){11}^(.75){10}&\undertag{} {e_1} {4pt} & \Circ
  \lineto[uu]^(.5){2} \lineto[rr]^(.25){7}^(.75){20} &\undertag{} {e_1} {4pt} &
\Circ \lineto[ur]^(.5){2}
  \lineto[dr]_(.5){3}&\\
&\Circ&&&&&&\overtag\Circ {w} {8pt}
\hbox to 0 pt{~,\hss} }$$ 
is not the splice diagram of any graph orbifold, even though it
satisfy the semigroup condition. The reason is that one has by the
edge determinant equation (Corollary 3.3 in \cite{myarticle}) that the
order of $H^{orb}_1(M)$ divides all edge determinants, so if $\Gamma$
where the splice diagram for some $M$, then $\num{H^{orb}_1(M)}$ would
divide $D(e_1)=26$ and $D(e_2)=20$, and hence divide $2$. Now $M$ can
not be an integer homology sphere because then all weight adjacent to
a node would have to be pairwise coprime according to
\cite{EisenbudNeumann}, so $H^{orb}_1(M)=\Z/2\Z$. Using the
topological description of the ideal generator given in section 12.1
in \cite{neumannandwahl2}, one easily sees that given any edge  $e$ in
$\Gamma(M)$, the product of the two ideal generators associated to
each of the ends of $e$ has to divide the order of $H_1^{orb}(M)$,
this includes edges to leaves. Now the ideal generator associated to
leaf $w$ is $4$, and hence do not divide the order of $H_1^{orb}(M)$
so we get a contradiction.

The above consideration on ideal generators leads to the following
necessary condition for a splice diagram $\Gamma$ to be realized from
a graph orbifold: The product of the ideal generators associated to
any edge has to divide all the edge determinants. 

But even this condition is not sufficient. The following splice diagram
satisfy it, but is not realizable by any graph orbifold. 
$$\splicediag{8}{30}{ 
&&&\Circ &&\Circ && \\
&&&& \overtag\Circ {v_3} {8pt}\lineto[ul]_(.5){2} \lineto[ur]^(.5){5} &&&\\
&&&&&&&\\
\Gamma'=&&&&&&&\\
&&&&&&&\\
&&&&\Circ \lineto[uuuu]_(.25){40}_(.75){165}&&&\\
&\Circ &&& &&&\Circ \\
 && \undertag\Circ {v_1} {4pt}
  \lineto[ul]_(.5){2} \lineto[uurr]^(.25){90}^(.75){42} 
\lineto[dl]^(.5){3} &&&& 
\undertag\Circ {v_2} {4pt}
  \lineto[ur]^(.5){3} \lineto[uull]_(.25){1722}_(.75){15} 
\lineto[dr]_(.5){5} &\\
&\Circ &&&&&&\Circ
\hbox to 0 pt{~.\hss} }$$  
The edge determinant equation and the above mentioned condition
implies that $\num{H^{orb}_1(M)}=30$ if $M$ is a graph orbifold
realizing $\Gamma'$. Then using this and the edge determinant equation
we can make a splice diagram for $M_{v_1e_1}$ since we know that
$\num{H^{orb}_1(M_{v_1e_1})}=90$. We can continue doing this until we
get that there exists a graph manifold $M'$ with $\num{H_1(M')}=60$
and with the following splice diagram.    
$$\splicediag{8}{30}{ 
 &\Circ & \\
&&\\
&&\\
 &\Circ \lineto[uuu]_(.5){10}
  \lineto[ddl]^(.5){6}
    \lineto[ddr]_(.5){21}&\\
&&\\
\Circ&&\Circ 
\hbox to 0 pt{~,\hss} }$$ 
we now this has to be manifold, since all the singular fibers has come
from the process of creating $M_{ve}$'s, and hence does not have
orbifold curves. This means $M'$ is a Seifert fibered manifold with
Seifert invariants $(1,-b),(6,\beta_1)(21,\beta_2),(10,\beta_3)$,
a simple calculation shows that such a Seifert fibered manifold can not
have first homology group of order $60$.

The failure of the last example is not as easy as the first to specify
in a nice condition, so do at the moment not have a good idea on a
set of necessary conditions for a splice diagram to be realized by
graph orbifold. Even without this, it might still be possible to used
Theorem \ref{linkascover} to prove it for more general graph orbifolds
that just the once having two node splice diagram. 

An other interesting question is, what are splice diagram
singularities coming from diagrams as above. 

\newpage

\bibliography{congruencecondition2nodecase}

\end{document}